\documentclass[11pt]{article}
\usepackage{amssymb,amsfonts,amsmath,latexsym,epsf,tikz,url}
\usepackage[usenames,dvipsnames]{pstricks}
\usepackage{pstricks-add}
\usepackage{epsfig}
\usepackage{pst-grad} 
\usepackage{pst-plot} 
\usepackage[space]{grffile} 
\usepackage{etoolbox} 
\makeatletter 
\patchcmd\Gread@eps{\@inputcheck#1 }{\@inputcheck"#1"\relax}{}{}
\makeatother

\newtheorem{theorem}{Theorem}[section]

\newtheorem{corollary}[theorem]{Corollary}
\newtheorem{observation}[theorem]{Observation}
\newtheorem{lemma}[theorem]{Lemma}
\newtheorem{remark}[theorem]{Remark}

\newcommand{\proof}{\noindent{\bf Proof.\ }}
\newcommand{\qed}{\hfill $\square$\medskip}

\begin{document}

\title{Enumeration of accurate dominating sets} 

\author{
Saeid Alikhani$^{1,}$\footnote{Corresponding author}
, Maryam Safazadeh$^1$, Nima Ghanbari$^2$
}

\date{\today}

\maketitle

\begin{center}
$^1$Department of Mathematics, Yazd University, 89195-741, Yazd, Iran\\
$^2$Department of Informatics, University of Bergen, P.O. Box 7803, 5020 Bergen, Norway

\medskip
{\tt alikhani@yazd.ac.ir, msafazadeh92@gmail.com, Nima.ghanbari@uib.no}
\end{center}

\begin{abstract}
Let $G=(V,E)$ be a simple graph. A dominating set of $G$ is a subset $D\subseteq V$ such that every vertex not in $D$ is adjacent to at least one vertex in $D$.
The cardinality of a smallest dominating set of $G$, denoted by $\gamma(G)$, is the domination number of $G$. A dominating set $D$ is an accurate dominating set of $G$, if 
no $|D|$-element subset of $V\setminus D$ is a dominating set of $G$.  The accurate domination number, $\gamma_a(G)$, is the cardinality of a smallest accurate dominating set  $D$. In this paper, after presenting preliminaries, we  count the number of accurate dominating sets of some specific graphs.
\end{abstract}


\medskip
\noindent{\bf AMS Subj.\ Class.}: 05C69, 05C05, 05C75

\section{Introduction }
Let $G = (V,E)$ be a simple graph with $n$ vertices. Throughout this paper we consider only simple graphs.  A set $D\subseteq V(G)$ is a  dominating set if every vertex in $V(G)\backslash D$ is adjacent to at least one vertex in $D$.
The  domination number $\gamma(G)$ is the minimum cardinality of a dominating set in $G$. There are various domination numbers in the literature.
For a detailed treatment of domination theory, the reader is referred to \cite{domination}.

An accurate dominating set of $G$ is a dominating set $D$ of $G$ such that no $|D|$-element subset of $V\setminus D$ is a dominating set of $G$. The accurate domination number of $G$, $\gamma_a(G)$, is the cardinality of a smallest accurate dominating set of $G$.  A dominating set
with cardinality $\gamma(G)$ is  called a {\it $\gamma$-set}.  Also an accurate dominating set of $G$ of cardinality $\gamma_a(G)$ is called a $\gamma_a$-set of $G$.
The accurate domination in graphs was introduced by Kulli and Kattimani \cite{11}, and further studied in a number of papers (see, for example, \cite{original,3}).

The concept of domination and related invariants have
been generalized in many ways. Among the best know generalizations are total, independent, and connected dominating, each of them with the corresponding domination number. Most of the papers published so far deal with structural
aspects of domination, trying to determine exact expressions for $\gamma(G)$  or some upper and/or lower bounds for it. There were no paper concerned with the 
enumerative side of the problem by 2008.

Regarding to enumerative side of dominating sets, Alikhani and  Peng \cite{saeid1}, have introduced the domination polynomial of a graph. The domination polynomial of graph $G$ is the  generating function for the number of dominating sets of  $G$, i.e., $D(G,x)=\sum_{ i=1}^{|V(G)|} d(G,i) x^{i}$ (see \cite{euro,saeid1}).   This  polynomial and its roots has been actively studied in recent
years (see for example \cite{Kot,Oboudi}). 
It is natural to count the number of another kind of dominating sets (\cite{utilitas}).     
Let ${\cal D}_a(G,i)$ be the family of
accurate  dominating sets of a graph $G$ with cardinality $i$ and let
$d_a(G,i)=|{\cal D}_a(G,i)|$. The generating function for the number of accurate dominating sets of $G$ is denoted by $D_a(G,x)$. 

The  corona of two graphs $G_1$ and $G_2$, is the graph
$G_1 \circ G_2$ formed from one copy of $G_1$ and $|V(G_1)|$ copies of $G_2$,
where the ith vertex of $G_1$ is adjacent to every vertex in the ith copy of $G_2$.
The corona $G\circ K_1$, in particular, is the graph constructed from a copy of $G$,
where for each vertex $v\in V(G)$, a new vertex $v'$ and a pendant edge $vv'$ are added.
The  join of two graphs $G_1$ and $G_2$, denoted by $G_1\vee G_2$,
is a graph with vertex set  $V(G_1)\cup V(G_2)$
and edge set $E(G_1)\cup E(G_2)\cup \{uv| u\in V(G_1)$ and $v\in V(G_2)\}$.

\medskip
In the next section, we consider specific graphs and count the number of their accurate dominating sets. In Sections 3 and 4, we study the problem of the number of accurate dominating sets of paths and cycles, respectively.

\section{Enumeration of accurate dominating sets of certain graphs}
By the definition of accurate dominating set, every accurate dominating set is a dominating set (and so $\gamma(G) \leq \gamma_a(G)$) but the converse is not true. In other words, in some graphs there exists dominating sets which are not accurate dominating sets. For example in the graph path $P_5$ with $V(P_5)=\{v_1,v_2,v_3,v_4,v_5\}$ (see Figure \ref{path}), the 
set $D=\{v_1,v_4\}$ is a dominating set of $P_5$ which is not accurate dominating set. As an example for accurate dominating set, we consider  the dominating set $D=\{v_2,v_4\}$  for $P_5$.

		\begin{figure}
		\begin{center}
			\psscalebox{0.8 0.8}
			{
\begin{pspicture}(0,-3.4135578)(5.194231,-2.6464422)
\psdots[linecolor=black, dotsize=0.4](0.19711548,-3.2164423)
\psdots[linecolor=black, dotsize=0.4](1.3971155,-3.2164423)
\psdots[linecolor=black, dotsize=0.4](2.5971155,-3.2164423)
\psdots[linecolor=black, dotsize=0.4](3.7971156,-3.2164423)
\psdots[linecolor=black, dotsize=0.4](4.9971156,-3.2164423)
\psline[linecolor=black, linewidth=0.08](0.19711548,-3.2164423)(4.9971156,-3.2164423)(4.9971156,-3.2164423)
\rput[bl](0.037115477,-2.9564424){$v_1$}
\rput[bl](1.2571155,-2.9564424){$v_2$}
\rput[bl](2.4371154,-2.8964422){$v_3$}
\rput[bl](3.6371155,-2.9164422){$v_4$}
\rput[bl](4.8171153,-2.9164422){$v_5$}
\end{pspicture}
}
		\end{center}
		\caption{The graph $P_5$ with $V(P_5)=\{v_1,v_2,v_3,v_4,v_5\}$.} \label{path}
	\end{figure}
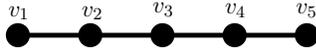

In this section, we study the number of accurate dominating sets of specific graphs. First we state some known results. 

\begin{lemma} {\rm \cite{original}}
	\begin{enumerate} 
		\item[(i)] 
		For every natural number $n$, $\gamma_a(K_n)=\lfloor\frac{n}{2}\rfloor+1$.
		\item[(ii)]
		For every natural number $n$, $\gamma_a(K_{n,n}) =n+1$. 
		\item[(iii)]   
		For $n>m\geq 1$, $\gamma_a(K_{m,n})=m$.
		\item[(iv)] 
		For $n\geq 3$, $\gamma_a(C_n)=\lfloor\frac{n}{3}\rfloor-\lfloor\frac{3}{n}\rfloor+2$.
		\item[(v)]   $\gamma_a(P_n) =\lceil\frac{n}{3}\rceil$ unless $n\in \{2,4\}$. 
	\end{enumerate} 
\end{lemma}

The following theorem is easy to obtain: 
\begin{theorem} 
	\begin{enumerate}
		\item [(i)]
		If $L_n$ is the ladder graph (the Cartesian product of $P_n$ and $K_2$), then $\gamma_a(L_n)=\lceil\frac{n}{2}\rceil+2.$
		
		\item[(ii)] 
		Let  $B_n$ ($n\geq 2$) be the book graph (Cartesian product of $K_{1,n}$ and $K_2$). Then $\gamma_a(B_2)=4$ and for $n\geq 3$,  $\gamma(B_n)=\gamma_a(B_n)=2$. 
		\item[(iii)] 
			For $n\geq 3$, $d_a(B_n,2)=1$ and for $3\leq i\leq \frac{n}{2}$, $d_a(B_n,i)=0.$
			
		\item [(iv)]
		If $Q_n$ is the hypercube  graph (Cartesian product of $Q_{n-1}$ and $K_2$), then $\gamma_a(Q_n)= 2^{n-1}+1$. 
		\item [(v)] 
		The number of accurate dominating sets of hypercube  graph  $Q_n$ with cardinality $i$ is  $d_a(Q_n,i)= {2^n\choose i}$.
		
	\end{enumerate} 
\end{theorem} 

\begin{theorem} \label{new} 
	Let $G$ be a graph of order $n$.
	\begin{enumerate} 
		\item[(i)] 
		If $D$ is a dominating set with $|D|\geq \lfloor \frac{n}{2}\rfloor+1$, then $D$ is an accurate dominating set of $G$.
		\item[(ii)] 
		$d(G,i)=d_a(G,i)$ for every $i\geq \lfloor \frac{n}{2}\rfloor+1$. 
		
		\item[(iii)] If $\gamma_a(G)>\lfloor \frac{n}{2}\rfloor$, then $D(G,x)=D_a(G,x)$.
	\end{enumerate}  
\end{theorem}
\proof
	\begin{enumerate} 
		\item[(i)] 
		Suppose that $D$ is a dominating set with  $|D|\geq \lfloor \frac{n}{2}\rfloor+1$. So $|V\setminus D|<\frac{n}{2}$  and no $|D|$-element subset of $V\setminus D$ is a dominating set of $G$. Therefore $D$ is an accurate dominating set of $G$.
		\item [(ii)] 
		By Part (i) every dominating sets and accurate dominating sets of $G$ are equal, for $i\geq \lfloor \frac{n}{2}\rfloor+1$, and so $d(G,i)=d_a(G,i)$. 
		
		\item[(iii)] It follows from Part (ii) and the definition of the domination polynomial and the accurate domination polynomial. \qed 
			\end{enumerate}
By Theorem \ref{new}, for a graph $G$ of order $n$, we shall find  $d_a(G,i)$ for 
$\gamma_a(G)\leq i\leq \lfloor \frac{n}{2}\rfloor$. Because for $i> \lfloor \frac{n}{2}\rfloor$, we can use results on the domination polynomial to obtain $d_a(G,i)$. 
The following theorem is about the accurate domination number and the number of accurate dominating sets of friendship graph which is the join of $K_1$ and $nK_2$.  

\begin{theorem} 
	Let $F_n$ be the friendship graph.
	\begin{enumerate} 
		\item[(i)] $\gamma_a(F_n)=1$.
		
		\item[(ii)]
		For every $i\in \mathbb{N}$, $d_a(F_n,i)={n\choose i-n}2^{n}+{2n\choose i-1}$. 
	\end{enumerate} 
\end{theorem} 
\proof 
\begin{enumerate} 
	\item[(i)] If $v$ is the central vertex of $F_n$, then $D=\{v\}$ is an accurate 
	dominating set of $F_n$, and so we have the result. 
	
	\item[(ii)]
	Let $D$ be an accurate dominating set of $F_n$ with cardinality $i$ and $v$ be the center vertex of $F_n$. There are two cases: 	
	
		\noindent {(1)} If $v\in D$,  then $D=D'\cup\{v\}$ and any $i-1$ vertices of $V(F_n)\setminus \{v\}$ can be in $D'$. So in this case the number of accurate dominating sets of $F_n$ with cardinality $i$ is ${2n\choose i-1}$.   
		
			\noindent {(2)} If $v\not\in D$,  then $i\geq n$ and we first choose $n$ vertices from each bases of triangles with $2^n$ cases and choose $i-n$ vertices with ${n\choose i-n}$ cases. Therefore we have the result. \qed
\end{enumerate}

We need the following theorem: 

\begin{theorem} {\rm\cite{Payan}}
	For a graph $G$ with even order $n$ and no isolated vertices, $\gamma(G)=\frac{n}{2}$ if 
	and only if the components of $G$ are the cycles $C_4$ or the corona $H\circ K_1$ for some connected graph $H$. 
\end{theorem} 

\begin{theorem} 
	If $G$ is a graph of order $n$, then $\gamma_a(G\circ K_1)=\gamma(G\circ K_1)+1=n+1$. 
\end{theorem} 
\proof
Suppose that $V(G)=\{v_1,...,v_n\}$ and the vertex $u_i$ is adjacent to $v_i$ for every $1\leq i\leq n$.  If $D$ is a dominating set of $G$, then for every $i$, $1 \leq i \leq n$, $u_i \in D$ or $v_i \in D$. This implies that $|D| \geq n$. Since $\lbrace u_1,\ldots,u_n\rbrace$ is a dominating set of $G \circ K_1$ with minimum cardinality, we have $\gamma(G\circ K_1)=n$. By adding one vertex to the set $D$, we have an accurate dominating set, so  $\gamma_a(G\circ K_1)=\gamma(G\circ K_1)+1=n+1$.\qed
\begin{corollary}
	Let $G$ be a  graph of order $n$.
	\begin{enumerate} 
		\item[(i)] 	For $n+1\leq m \leq 2n$,
		we have $d_a(G\circ K_1,m)={n \choose m-n}2^{2n-m}$.
		\item[(ii)] 	
		$D_a(G\circ K_1,x)=x^n(x+2)^n-2x^n.$
	\end{enumerate}  
\end{corollary}
\proof 
\begin{enumerate} 
	\item[(i)]
	Suppose that $D$ is an accurate dominating set of $G\circ K_1$ of size $m$, and $ |D \cap V(G)|=i$, for $0 \leq i \leq n$. Without loss of generality, suppose that $ D \cap V(G)=\lbrace v_1,\ldots,v_i\rbrace$, so  $D$ contains $u_{i+1},\ldots,u_n$ and the vertex $u_{n+1}$ to be an accurate dominating set. For extending $\lbrace v_1,\cdots,v_i,u_{i+1},u_{i+2},\cdots,u_n,u_{n+1} \rbrace$ to an accurate dominating set $D$ of size $m$ with $ D \cap V(G)=\lbrace v_1,\cdots,v_i\rbrace$, we have ${n \choose m-n}$ possibilities. Therefore 
	$d_a(G\circ K_1,m)=\sum_{i=0}^{n}{n \choose i}{i\choose m-n}$.
	It is not hard to see that the above sum is equal to ${n \choose m-n}2^{2n-m}$.

\item[(ii)] Since there is no accurate dominating sets of cardinality $i\leq n$ for $G\circ K_1$, so we have the result from Part (i). \qed
\end{enumerate} 


\section{Counting the number of accurate dominating sets of $P_n$}

In this section, we want to count the number of accurate dominating sets of path graph $P_n$. We need the following theorem: 

\begin{theorem}{\rm\cite{Saeid}}
	The number of dominating sets of path $P_n$ satisfies  the following recursive relation: 
	\[d(P_n,i)= d(P_{n-1},i-1)+d(P_{n-2},i-1)+d(P_{n-3},i-1).\]
	
\end{theorem}

The following theorem gives the explicit formula for the number of dominating sets of $P_n$ (\cite{Llano}):

\begin{theorem} \label{Llano1}
	For every $n\geq 1$, $d(P_n,k)=\displaystyle\sum_{m=0}^{\lfloor\frac{n-k}{2}\rfloor+1}{k-1 \choose n-k-m}{n-k-m+2 \choose m}.$ 
\end{theorem} 

By Theorem \ref{new}(ii), we know that $d(P_n,i)=d_a(P_n,i)$, for $i\geq \lfloor\frac{n}{2}\rfloor+1$. So to study the number of accurate dominating sets of $P_n$, we need to consider $d_a(P_n,i)$ for $\lceil\frac{n}{3}\rceil\leq i \leq \lfloor\frac{n}{2}\rfloor$.

\begin{theorem} \label{path-lower}
The number of accurate dominating sets of path $P_n$ with cardinality $i$, where $\lceil\frac{n}{3}\rceil\leq i \leq \lfloor\frac{n}{2}\rfloor$,  satisfies: 
	\begin{align*}
	d_a(P_n,i) \geq 
	\displaystyle\sum_{k=3}^{i}d(P_{n-k},i-k+1)+ \displaystyle\sum_{k=5}^{i+1}d(P_{n-k},i-k+2).
	\end{align*}
\end{theorem}

	\begin{proof} 
	Let $V(P_n)=\{v_1,v_2,....,v_n\}$ (see Figure \ref{path1}) and  $D$ be an accurate  dominating  set of $P_n$ with cardinality $i$, where $\lceil\frac{n}{3}\rceil\leq i \leq \lfloor\frac{n}{2}\rfloor$. At least one of the vertices $v_1$ or $v_2$ should be in the dominating set. So for finding a lower bound of $d_a(P_n,i)$, we consider the following cases: 
	\begin{enumerate}
	\item[(i)]
	For every $3\leq k\leq i$, if $v_1,v_2,v_3,\ldots,v_{k-1}\in D$ and $v_{k}\not\in D$, then 
	$${\cal D}_a(P_n,i)={\cal D}(P_{n-k},i-k+1)\cup \{v_1,v_2,v_3,\ldots,v_{k-1}\}.$$
			 In this case the number of accurate dominating sets of $P_n$ with cardinality $i$ is $d(P_{n-k},i-k+1)$.
	\item[(ii)]
		For every $4\leq k\leq i$, if $v_2,v_3,v_4,\ldots,v_{k}\in D$ and $v_1,v_{k+1}\not\in D$,  
			 $${\cal D}_a(P_n,i)={\cal D}(P_{n-k},i-k+2)\cup \{v_2,v_3,v_4,\ldots,v_{k}\}.$$
			 In this case the number of accurate dominating sets of $P_n$ with cardinality $i$ is $d(P_{n-k},i-k+2)$.
	\end{enumerate}
	Therefore we have the result.\qed
	\end{proof}

		\begin{figure}
		\begin{center}
			\psscalebox{0.56 0.56}
{
\begin{pspicture}(0,-9.07)(13.05,4.89)
\psdots[linecolor=black, dotsize=0.4](1.4,4.18)
\psdots[linecolor=black, dotsize=0.4](2.2,4.18)
\psdots[linecolor=black, dotsize=0.4](3.0,4.18)
\psdots[linecolor=black, dotsize=0.4](3.8,4.18)
\psdots[linecolor=black, dotsize=0.4](4.6,4.18)
\psdots[linecolor=black, dotsize=0.4](5.4,4.18)
\psline[linecolor=black, linewidth=0.08](1.0,2.98)(6.2,2.98)(6.2,2.98)
\psline[linecolor=black, linewidth=0.08](1.0,2.18)(1.0,2.18)(6.2,2.18)(6.2,2.18)
\psline[linecolor=black, linewidth=0.08](1.0,1.38)(6.2,1.38)(6.2,1.38)
\psline[linecolor=black, linewidth=0.08](1.0,0.58)(6.2,0.58)(6.2,0.58)
\psdots[linecolor=black, dotsize=0.4](1.4,2.58)
\psdots[linecolor=black, dotsize=0.4](2.2,1.78)
\psdots[linecolor=black, dotsize=0.4](2.2,0.98)
\psdots[linecolor=black, dotsize=0.4](1.4,0.98)
\psline[linecolor=black, linewidth=0.08](1.0,-1.02)(1.0,2.98)(1.0,2.98)
\psline[linecolor=black, linewidth=0.08](1.0,-1.02)(6.2,-1.02)(6.2,-1.02)
\psline[linecolor=black, linewidth=0.08](1.0,-0.22)(1.0,-2.62)(1.0,-2.62)
\psline[linecolor=black, linewidth=0.08](1.0,-2.62)(6.2,-2.62)(6.2,-2.62)
\psline[linecolor=black, linewidth=0.08](1.0,-1.82)(6.2,-1.82)(6.2,-1.82)
\psdots[linecolor=black, dotsize=0.4](2.2,-2.22)
\psdots[linecolor=black, dotsize=0.4](3.0,-2.22)
\psdots[linecolor=black, dotsize=0.4](2.2,-3.02)
\psdots[linecolor=black, dotsize=0.4](3.0,-3.02)
\psdots[linecolor=black, dotsize=0.4](3.8,-3.02)
\psline[linecolor=black, linewidth=0.08](5.4,-3.42)(6.2,-3.42)(6.2,-3.42)
\psdots[linecolor=black, dotsize=0.1](2.2,-4.62)
\psdots[linecolor=black, dotsize=0.1](2.2,-5.02)
\psdots[linecolor=black, dotsize=0.1](2.2,-5.42)
\psline[linecolor=black, linewidth=0.08](1.0,-5.02)(1.0,-5.82)(6.2,-5.82)(6.2,-5.82)
\psdots[linecolor=black, dotsize=0.4](2.2,-1.42)
\psdots[linecolor=black, dotsize=0.4](3.0,-1.42)
\psdots[linecolor=black, dotsize=0.4](3.8,-1.42)
\psdots[linecolor=black, dotsize=0.4](4.6,-1.42)
\psdots[linecolor=black, dotsize=0.4](5.4,-1.42)
\psdots[linecolor=black, dotsize=0.4](2.2,-6.22)
\psdots[linecolor=black, dotsize=0.4](3.0,-6.22)
\psdots[linecolor=black, dotsize=0.4](3.8,-6.22)
\psdots[linecolor=black, dotsize=0.4](4.6,-6.22)
\psdots[linecolor=black, dotsize=0.4](5.4,-6.22)
\psline[linecolor=black, linewidth=0.08](1.0,-5.82)(1.0,-6.62)(6.2,-6.62)(6.2,-6.62)
\psdots[linecolor=black, dotsize=0.4](11.8,4.18)
\psdots[linecolor=black, dotsize=0.4](12.6,4.18)
\psline[linecolor=black, linewidth=0.08](6.2,2.98)(13.0,2.98)(13.0,-6.62)(6.2,-6.62)(6.2,-6.62)
\psline[linecolor=black, linewidth=0.08](5.8,-5.82)(13.0,-5.82)(13.0,-5.82)
\psline[linecolor=black, linewidth=0.08](6.2,-3.42)(13.0,-3.42)(13.0,-3.42)
\psline[linecolor=black, linewidth=0.08](6.2,-2.62)(13.0,-2.62)(13.0,-2.62)
\psline[linecolor=black, linewidth=0.08](6.2,-1.82)(13.0,-1.82)(13.0,-1.82)
\psline[linecolor=black, linewidth=0.08](6.2,-1.02)(13.0,-1.02)(13.0,-1.02)
\psline[linecolor=black, linewidth=0.08](6.2,0.58)(13.0,0.58)(13.0,0.58)
\psline[linecolor=black, linewidth=0.08](6.2,1.38)(13.0,1.38)(13.0,1.38)
\psline[linecolor=black, linewidth=0.08](5.8,2.18)(13.0,2.18)(13.0,2.18)
\rput[bl](1.22,4.58){${v_1}$}
\rput[bl](2.04,4.6){${v_2}$}
\rput[bl](2.78,4.6){${v_3}$}
\rput[bl](3.6,4.62){${v_4}$}
\rput[bl](4.46,4.64){${v_5}$}
\rput[bl](5.24,4.62){${v_6}$}
\rput[bl](12.46,4.46){${v_n}$}
\rput[bl](11.48,4.42){${v_{n-1}}$}
\psdots[linecolor=black, dotsize=0.1](6.2,-1.42)
\psdots[linecolor=black, dotsize=0.1](6.6,-1.42)
\psdots[linecolor=black, dotsize=0.1](7.0,-1.42)
\psdots[linecolor=black, dotsize=0.1](6.2,-6.22)
\psdots[linecolor=black, dotsize=0.1](6.6,-6.22)
\psdots[linecolor=black, dotsize=0.1](7.0,-6.22)
\psdots[linecolor=black, dotsize=0.4](7.8,-1.42)
\psdots[linecolor=black, dotsize=0.4](7.8,-6.22)
\psdots[linecolor=black, dotsize=0.4](8.6,-6.22)
\psdots[linecolor=black, fillstyle=solid, dotstyle=o, dotsize=0.4, fillcolor=white](8.6,-1.42)
\psdots[linecolor=black, fillstyle=solid, dotstyle=o, dotsize=0.4, fillcolor=white](9.4,-6.22)
\psline[linecolor=black, linewidth=0.08](9.8,-5.82)(9.8,-6.62)(9.8,-6.62)
\psdots[linecolor=black, fillstyle=solid, dotstyle=o, dotsize=0.4, fillcolor=white](1.4,-6.22)
\psdots[linecolor=black, dotsize=0.4](1.4,-1.42)
\rput[bl](0.42,2.52){1}
\rput[bl](0.38,1.7){2}
\rput[bl](0.36,0.9){3}
\psline[linecolor=black, linewidth=0.06](1.4,4.18)(5.8,4.18)(5.8,4.18)
\psline[linecolor=black, linewidth=0.06](11.4,4.18)(12.6,4.18)(12.6,4.18)
\psdots[linecolor=black, dotsize=0.4](9.4,4.18)
\psdots[linecolor=black, dotsize=0.4](8.6,4.18)
\psdots[linecolor=black, dotsize=0.4](7.8,4.18)
\psline[linecolor=black, linewidth=0.06](7.4,4.18)(9.8,4.18)(9.8,4.18)
\psdots[linecolor=black, dotsize=0.1](6.2,4.18)
\psdots[linecolor=black, dotsize=0.1](6.6,4.18)
\psdots[linecolor=black, dotsize=0.1](7.0,4.18)
\psdots[linecolor=black, dotsize=0.1](10.2,4.18)
\psdots[linecolor=black, dotsize=0.1](10.6,4.18)
\psdots[linecolor=black, dotsize=0.1](11.0,4.18)
\rput[bl](8.44,4.52){${v_i}$}
\rput[bl](7.52,4.52){${v_{i-1}}$}
\rput[bl](9.12,4.5){${v_{i+1}}$}
\psdots[linecolor=black, dotsize=0.4](2.2,2.58)
\psdots[linecolor=black, dotsize=0.4](1.4,1.78)
\psdots[linecolor=black, dotsize=0.4](3.0,1.78)
\psdots[linecolor=black, dotsize=0.4](3.0,0.98)
\psdots[linecolor=black, dotsize=0.4](3.8,0.98)
\psdots[linecolor=black, dotsize=0.1](2.2,0.18)
\psdots[linecolor=black, dotsize=0.1](2.2,-0.22)
\psdots[linecolor=black, dotsize=0.1](2.2,-0.62)
\psline[linecolor=black, linewidth=0.08](1.0,-2.22)(1.0,-5.82)(1.0,-5.82)
\psline[linecolor=black, linewidth=0.08](1.0,-3.42)(5.4,-3.42)(5.4,-3.42)
\psline[linecolor=black, linewidth=0.08](1.0,-4.22)(13.0,-4.22)(13.0,-4.22)
\psdots[linecolor=black, dotsize=0.4](3.8,-2.22)
\psdots[linecolor=black, dotsize=0.4](4.6,-3.02)
\psdots[linecolor=black, dotsize=0.4](2.2,-3.82)
\psdots[linecolor=black, dotsize=0.4](3.0,-3.82)
\psdots[linecolor=black, dotsize=0.4](3.8,-3.82)
\psdots[linecolor=black, dotsize=0.4](4.6,-3.82)
\psdots[linecolor=black, dotsize=0.4](5.4,-3.82)
\psdots[linecolor=black, dotstyle=o, dotsize=0.4, fillcolor=white](3.0,2.58)
\psdots[linecolor=black, dotstyle=o, dotsize=0.4, fillcolor=white](3.8,1.78)
\psdots[linecolor=black, dotstyle=o, dotsize=0.4, fillcolor=white](4.6,0.98)
\psdots[linecolor=black, dotstyle=o, dotsize=0.4, fillcolor=white](1.4,-2.22)
\psdots[linecolor=black, dotstyle=o, dotsize=0.4, fillcolor=white](1.4,-3.02)
\psdots[linecolor=black, dotstyle=o, dotsize=0.4, fillcolor=white](1.4,-3.82)
\psdots[linecolor=black, dotstyle=o, dotsize=0.4, fillcolor=white](4.6,-2.22)
\psdots[linecolor=black, dotstyle=o, dotsize=0.4, fillcolor=white](5.4,-3.02)
\psdots[linecolor=black, dotstyle=o, dotsize=0.4, fillcolor=white](6.2,-3.82)
\psline[linecolor=black, linewidth=0.08](3.4,2.98)(3.4,2.18)(4.2,2.18)(4.2,1.38)(5.0,1.38)(5.0,0.58)(5.0,0.58)
\psline[linecolor=black, linewidth=0.08](9.0,-1.02)(9.0,-1.82)(9.0,-1.82)
\psline[linecolor=black, linewidth=0.08](5.0,-1.82)(5.0,-2.62)(5.8,-2.62)(5.8,-3.42)(6.6,-3.42)(6.6,-4.22)(6.6,-4.22)
\rput[bl](0.32,-1.54){i-2}
\rput[bl](0.32,-2.34){i-1}
\rput[bl](0.46,-3.1){i}
\rput[bl](0.12,-3.96){i+1}
\rput[bl](0.04,-6.34){2i-5}
\psline[linecolor=black, linewidth=0.08](1.0,-6.62)(1.0,-9.02)(13.0,-9.02)(13.0,-6.62)(13.0,-6.62)
\psline[linecolor=black, linewidth=0.08](1.0,-7.42)(13.0,-7.42)(13.0,-7.42)
\psline[linecolor=black, linewidth=0.08](1.0,-8.22)(13.0,-8.22)(13.0,-8.22)
\psdots[linecolor=black, dotstyle=o, dotsize=0.4, fillcolor=white](1.4,-7.82)
\psdots[linecolor=black, dotstyle=o, dotsize=0.4, fillcolor=white](3.0,-7.82)
\psdots[linecolor=black, dotstyle=o, dotsize=0.4, fillcolor=white](1.4,-8.62)
\psdots[linecolor=black, dotstyle=o, dotsize=0.4, fillcolor=white](3.8,-8.62)
\psdots[linecolor=black, dotsize=0.4](1.4,-7.02)
\psdots[linecolor=black, dotsize=0.4](2.2,-7.82)
\psdots[linecolor=black, dotsize=0.4](2.2,-8.62)
\psdots[linecolor=black, dotsize=0.4](3.0,-8.62)
\psdots[linecolor=black, dotstyle=o, dotsize=0.4, fillcolor=white](2.2,-7.02)
\psline[linecolor=black, linewidth=0.08](2.6,-6.62)(2.6,-7.42)(3.4,-7.42)(3.4,-8.22)(4.2,-8.22)(4.2,-9.02)(4.2,-9.02)
\rput[bl](0.0,-7.16){2i-4}
\rput[bl](0.0,-7.88){2i-3}
\rput[bl](0.02,-8.76){2i-2}
\end{pspicture}
}
		\end{center}
		\caption{\small Making accurate dominating sets of $P_n$ related to Theorems \ref{path-lower} and  \ref{path-upper}} \label{path1}
	\end{figure}
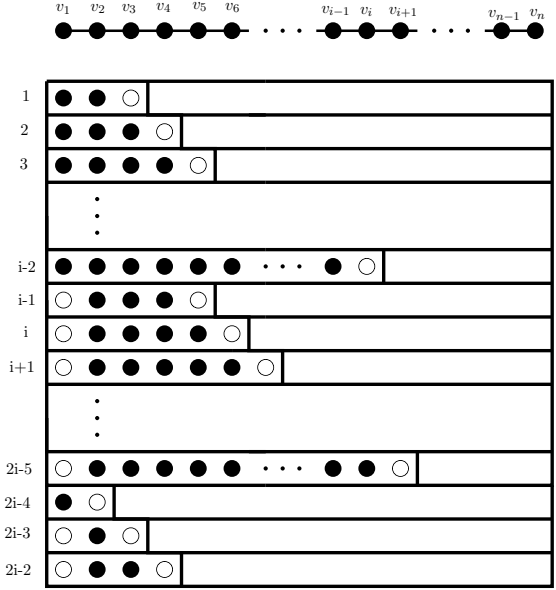

\begin{theorem} \label{path-upper}
The number of accurate dominating sets of path $P_n$ with cardinality $i$, where $\lceil\frac{n}{3}\rceil\leq i \leq \lfloor\frac{n}{2}\rfloor$, satisfies: 
	\begin{align*}
	d_a(P_n,i) &\leq d_a(P_{n-2},i-1) + d_a(P_{n-3},i-1) + d_a(P_{n-4},i-2) \\
	&\quad+  \displaystyle\sum_{k=3}^{i}d(P_{n-k},i-k+1)+ \displaystyle\sum_{k=5}^{i+1}d(P_{n-k},i-k+2).
	\end{align*}
\end{theorem}

	\begin{proof} 
		Let $V(P_n)=\{v_1,v_2,...,v_n\}$ (see Figure \ref{path1}) and  $D$ be an accurate  dominating  set of $P_n$ with cardinality $i$. At least one of the vertices $v_1$ or $v_2$ should be in the dominating set. So for finding an upper bound of $d_a(P_n,i)$, we consider the following cases: 
	\begin{enumerate}
	\item[(i)]
	For every $3\leq k\leq i$, if $v_1,v_2,v_3,\ldots,v_{k-1}\in D$ and $v_{k}\not\in D$, then 
	$${\cal D}_a(P_n,i)={\cal D}(P_{n-k},i-k+1)\cup \{v_1,v_2,v_3,\ldots,v_{k-1}\}.$$
			 In this case the number of accurate dominating sets of $P_n$ with cardinality $i$ is $d(P_{n-k},i-k+1)$.
	\item[(ii)]
		For every $4\leq k\leq i$, if $v_2,v_3,v_4,\ldots,v_{k}\in D$ and $v_1,v_{k+1}\not\in D$, then  
			 $${\cal D}_a(P_n,i)={\cal D}(P_{n-k},i-k+2)\cup \{v_2,v_3,v_4,\ldots,v_{k}\}.$$
			 In this case the number of accurate dominating sets of $P_n$ with cardinality $i$ is $d(P_{n-k},i-k+2)$.
	\item[(iii)]
		If $v_1\in D$ and $v_2\not\in D$, then ${\cal D}_a(P_n,i)\subseteq {\cal D}_a(P_{n-2},i-1)\cup \{v_1\}$. In this case $d_a(P_n,i)$ is at most $d_a(P_{n-2},i-1)$.  
	\item[(iv)]
		If $v_1\not\in D, v_2\in D$ and $v_3\not\in D$, then ${\cal D}_a(P_n,i)\subseteq {\cal D}_a(P_{n-3},i-1)\cup \{v_2\}$. In this case the number of accurate dominating sets of $P_n$ with cardinality $i$ is at most $d_a(P_{n-3},i-1)$.  
	\item[(v)]
		If $v_2,v_3\in D,$ and $v_1,v_4\not\in D$, then ${\cal D}_a(P_n,i)\subseteq {\cal D}_a(P_{n-4},i-2)\cup \{v_2,v_3\}$. In this case the number of accurate dominating sets of $P_n$ with cardinality $i$ is at most $d_a(P_{n-4},i-2)$. 
	\end{enumerate}
	Therefore we have the result.\qed
	\end{proof}

	\begin{remark}
	The upper bound in Theorem \ref{path-upper} is sharp. It suffices to consider path $P_7$ and $i=4$. Then we have $d_a(P_5,3)=8$, $d_a(P_4,3)=4$, $d_a(P_3,2)=3$, $d(P_4,2)=4$, $d(P_3,1)=1$, $d(P_2,1)=2$ and $d_a(P_7,4)=22$.
	\end{remark}

\begin{theorem} \label{path-lower2}
The number of accurate dominating sets of path $P_n$ with cardinality $i$, $d_a(P_n,i)$, satisfies: 
	\begin{align*}
	d_a(P_n,i) \geq 2d(P_{n-3},i-2) - d(P_{n-6},i-4). 
	\end{align*}
\end{theorem}

	\begin{proof} 
	Let $V(P_n)=\{v_1,v_2,...,v_n\}$ (see Figure \ref{path2}) and  $D$ be an accurate  dominating  set of $P_n$ with cardinality $i$. At least one of the vertices $v_1$ or $v_2$ should be in the dominating set $D$. So for finding a lower bound, we consider the following cases: 
	\begin{enumerate}
	\item[(i)]
	If $v_1,v_2\in D$ and $v_3\not\in D$, then  ${\cal D}_a(P_n,i) = {\cal D}(P_{n-3},i-2)\cup \{v_1,v_2\}$. In this case the number of accurate dominating sets of $P_n$ with cardinality $i$ is $d(P_{n-3},i-2)$. 
	\item[(ii)]
	If $v_{n-1},v_n\in D$ and $v_{n-2}\not\in D$, then  ${\cal D}_a(P_n,i) = {\cal D}(P_{n-3},i-2)\cup \{v_{n-1},v_n\}$. In this case the number of accurate dominating sets of $P_n$ with cardinality $i$ is $d(P_{n-3},i-2)$. 		
	\end{enumerate}	
Now, we consider the case  	$v_1,v_2,v_{n-1},v_n\in D$ and $v_3,v_{n-2}\not\in D$. 
	Since we count these cases twice, then  we should delete them once and the number of cases are $d(P_{n-6},i-4)$. Therefore we have the result.\qed
	\end{proof} 
	
		\begin{figure}
		\begin{center}
			\psscalebox{0.8 0.8}
{
\begin{pspicture}(0,-4.29)(7.89,0.03)
\psdots[linecolor=black, dotsize=0.4](1.04,-0.64)
\psdots[linecolor=black, dotsize=0.4](1.84,-0.64)
\psdots[linecolor=black, dotsize=0.4](2.64,-0.64)
\psline[linecolor=black, linewidth=0.08](0.64,-1.84)(5.84,-1.84)(5.84,-1.84)
\psline[linecolor=black, linewidth=0.08](0.64,-2.64)(0.64,-2.64)(5.84,-2.64)(5.84,-2.64)
\psline[linecolor=black, linewidth=0.08](0.64,-3.44)(5.84,-3.44)(5.84,-3.44)
\psline[linecolor=black, linewidth=0.08](0.64,-4.24)(5.84,-4.24)(5.84,-4.24)
\psdots[linecolor=black, dotsize=0.4](1.04,-2.24)
\psdots[linecolor=black, dotsize=0.4](1.84,-3.84)
\psdots[linecolor=black, dotsize=0.4](1.04,-3.84)
\rput[bl](0.86,-0.24){${v_1}$}
\rput[bl](1.68,-0.22){${v_2}$}
\rput[bl](2.42,-0.22){${v_3}$}
\rput[bl](7.28,-0.34){${v_n}$}
\rput[bl](6.34,-0.36){${v_{n-1}}$}
\rput[bl](0.06,-2.3){1}
\rput[bl](0.02,-3.12){2}
\rput[bl](0.0,-3.92){3}
\psdots[linecolor=black, dotsize=0.4](1.84,-2.24)
\psdots[linecolor=black, dotstyle=o, dotsize=0.4, fillcolor=white](2.64,-2.24)
\psline[linecolor=black, linewidth=0.08](3.04,-1.84)(3.04,-2.64)(3.04,-2.64)
\psline[linecolor=black, linewidth=0.08](3.04,-3.44)(3.04,-4.24)(3.04,-4.24)
\psline[linecolor=black, linewidth=0.08](1.44,-4.24)(0.64,-4.24)(0.64,-1.84)(3.44,-1.84)(3.44,-1.84)
\psline[linecolor=black, linewidth=0.08](1.04,-0.64)(3.44,-0.64)(3.44,-0.64)
\psdots[linecolor=black, dotsize=0.1](3.84,-0.64)
\psdots[linecolor=black, dotsize=0.1](4.24,-0.64)
\psdots[linecolor=black, dotsize=0.1](4.64,-0.64)
\psline[linecolor=black, linewidth=0.08](5.04,-0.64)(7.44,-0.64)(7.44,-0.64)
\psdots[linecolor=black, dotsize=0.4](5.84,-0.64)
\psdots[linecolor=black, dotsize=0.4](6.64,-0.64)
\psdots[linecolor=black, dotsize=0.4](7.44,-0.64)
\psline[linecolor=black, linewidth=0.08](4.64,-1.84)(7.84,-1.84)(7.84,-4.24)(5.84,-4.24)(5.84,-4.24)
\psline[linecolor=black, linewidth=0.08](5.84,-3.44)(7.84,-3.44)(7.84,-2.64)(5.84,-2.64)(5.44,-2.64)(5.44,-3.44)(5.44,-3.44)
\psline[linecolor=black, linewidth=0.08](5.44,-3.44)(5.44,-4.24)(5.44,-4.24)
\psdots[linecolor=black, dotsize=0.4](7.44,-3.04)
\psdots[linecolor=black, dotsize=0.4](6.64,-3.04)
\psdots[linecolor=black, dotsize=0.4](6.64,-3.84)
\psdots[linecolor=black, dotsize=0.4](7.44,-3.84)
\psdots[linecolor=black, dotstyle=o, dotsize=0.4, fillcolor=white](2.64,-3.84)
\psdots[linecolor=black, dotstyle=o, dotsize=0.4, fillcolor=white](5.84,-3.04)
\psdots[linecolor=black, dotstyle=o, dotsize=0.4, fillcolor=white](5.84,-3.84)
\rput[bl](5.48,-0.34){$v_{n-2}$}
\end{pspicture}
}
		\end{center}
		\caption{\small Making accurate dominating sets of $P_n$ related to the proof of Theorem \ref{path-lower2}} \label{path2}
	\end{figure}
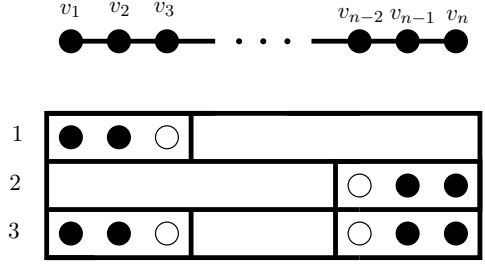

	\begin{remark}
	The lower bound in Theorem \ref{path-lower2} is sharp. It suffices to consider path $P_6$ (see Figure \ref{path6}) and $i=3$. Then we have two accurate dominating sets $\{1,2,5\}$ and $\{2,5,6\}$ and $d(P_3,1)=1$.
	\end{remark}
	
		\begin{figure}
		\begin{center}
			\psscalebox{0.8 0.8}
{
\begin{pspicture}(0,-2.9835577)(6.394231,-2.1564422)
\psdots[linecolor=black, dotsize=0.4](0.19711548,-2.7864423)
\psdots[linecolor=black, dotsize=0.4](1.3971155,-2.7864423)
\psdots[linecolor=black, dotsize=0.4](2.5971155,-2.7864423)
\psdots[linecolor=black, dotsize=0.4](3.7971156,-2.7864423)
\psdots[linecolor=black, dotsize=0.4](4.9971156,-2.7864423)
\psdots[linecolor=black, dotsize=0.4](6.1971154,-2.7864423)
\psline[linecolor=black, linewidth=0.08](0.19711548,-2.7864423)(6.1971154,-2.7864423)(6.1971154,-2.7864423)
\rput[bl](0.117115475,-2.5064423){1}
\rput[bl](1.2971154,-2.4664423){2}
\rput[bl](2.5171156,-2.4464424){3}
\rput[bl](3.6971154,-2.4264421){4}
\rput[bl](4.9171157,-2.4464424){5}
\rput[bl](6.0771155,-2.4264421){6}
\end{pspicture}
}
		\end{center}
		\caption{Path graph $P_6$ with vertex set $V=\{1,2,3,4,5,6\}$} \label{path6}
	\end{figure}
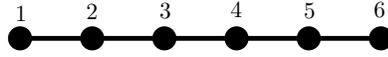

\section{Counting the number of accurate dominating sets of $C_n$}

 In this section we  count the number of accurate dominating sets of cycle graph $C_n$. We need the following theorem: 
 
 \begin{theorem}{\rm\cite{Opus}}
 	The number of dominating sets of cycle $C_n$ satisfies  the following recursive relation: 
 	\[d(C_n,i)= d(C_{n-1},i-1)+d(C_{n-2},i-1)+d(C_{n-3},i-1).\]
 	
 \end{theorem}
 
 The following theorem gives the explicit formula for the number of dominating sets of $C_n$ (\cite{Llano}):
 
 \begin{theorem} {\rm\cite{Llano}}\label{Llano2}
 	For every $n\geq 3$, $$d(C_n,k)=\displaystyle\sum_{m=0}^{\lfloor\frac{n-k}{2}\rfloor+1}{k-1 \choose n-k-m}\Big({n-k-m+2 \choose m+2}{n-k-m\choose m-2}\Big).$$ 
 \end{theorem}

By Theorem \ref{new}(ii), we know that $d(C_n,i)=d_a(C_n,i)$, for $i\geq \lfloor\frac{n}{2}\rfloor+1$. So to study the number of accurate dominating sets of $C_n$, we should consider $d_a(C_n,i)$ for $\lfloor\frac{n}{3}\rfloor+2\leq i \leq \lfloor\frac{n}{2}\rfloor$. 
We need the following observation: 

\begin{observation} \label{obs}
	If there is no three or more consecutive vertices in the dominating set $D$ of the cycle $C_n$ with $|D|\leq \frac{n}{2}$, then $D$ is not an accurate dominating set for $C_n$.
\end{observation}

\begin{theorem} \label{cycle-lower}
The number of accurate dominating sets of cycle $C_n$, $n\geq 6$, with cardinality $i$, where $\lfloor\frac{n}{3}\rfloor+2\leq i \leq \lfloor\frac{n}{2}\rfloor$, satisfies: 
	\begin{align*}
	d_a(C_n,i) \geq 
	\displaystyle\sum_{k=3}^{i-1}(k+2)d(P_{n-k-2},i-k).
	\end{align*}
\end{theorem}

	\begin{proof} 
Let $V(C_n)=\{v_1,v_2,....,v_n\}$ (see Figure \ref{cycle}) and  $D$ be an accurate  dominating  set of $C_n$ with cardinality $i$, where $\lfloor\frac{n}{3}\rfloor+2\leq i \leq \lfloor\frac{n}{2}\rfloor$. By Observation \ref{obs},  at least one of the three vertices $v_j,v_{j+1},v_{j+2}$ ($1\leq  j \leq n-2$) are in $D$. Therefore one of the three vertices $v_1,v_{2},v_{n}$ are in $D$. We consider the following cases: 
	
		\medskip			
	\noindent {\bf Case (1)} 
	If $v_1,v_2,v_3\in D$ and $v_4,v_n\not\in D$, then ${\cal D}_a(C_n,i)={\cal D}(P_{n-5},i-3)\cup \{v_1,v_2,v_3\}$. In this case $d_a(C_n,i)$ is $d(P_{n-5},i-3)$.  
	
	\medskip			
	\noindent {\bf Case (2)} 
	If $v_1,v_2,v_n\in D$ and $v_3,v_{n-1}\not\in D$, then ${\cal D}_a(C_n,i)={\cal D}(P_{n-5},i-3)\cup \{v_1,v_2,v_n\}$. In this case $d_a(C_n,i)$ is $d(P_{n-5},i-3)$.  
	
	\medskip			
	\noindent {\bf Case (3)} 
	If $v_1,v_{n-1},v_n\in D$ and $v_2,v_{n-2}\not\in D$, then ${\cal D}_a(C_n,i)={\cal D}(P_{n-5},i-3)\cup \{v_1,v_{n-1},v_n\}$. In this case $d_a(C_n,i)$ is $d(P_{n-5},i-3)$. 
	
	\medskip			
	\noindent {\bf Case (4)} 
	If $v_2,v_3,v_4\in D$ and $v_1,v_{5}\not\in D$, then ${\cal D}_a(C_n,i)={\cal D}(P_{n-5},i-3)\cup \{v_2,v_3,v_4\}$. In this case $d_a(C_n,i)$ is $d(P_{n-5},i-3)$. 	
	
	\medskip			
	\noindent {\bf Case (5)} 
	If $v_{n-2},v_{n-1},v_n\in D$ and $v_1,v_{n-3}\not\in D$, then ${\cal D}_a(C_n,i)={\cal D}(P_{n-5},i-3)\cup \{v_{n-2},v_{n-1},v_n\}$. In this case $d_a(C_n,i)$ is $d(P_{n-5},i-3)$. 
	
	\medskip			
	\noindent {\bf Case (6)} 
	If $v_1,v_2,v_3,v_4\in D$ and $v_5,v_n\not\in D$, then ${\cal D}_a(C_n,i)={\cal D}(P_{n-6},i-4)\cup \{v_1,v_2,v_3,v_4\}$. In this case $d_a(C_n,i)$ is $d(P_{n-6},i-4)$. 
	
	\medskip			
	\noindent {\bf Case (7)} 
	If $v_1,v_2,v_3,v_n\in D$ and $v_4,v_{n-1}\not\in D$, then ${\cal D}_a(C_n,i)={\cal D}(P_{n-6},i-4)\cup \{v_1,v_2,v_3,v_n\}$. In this case $d_a(C_n,i)$ is $d(P_{n-6},i-4)$. 
	
	\medskip			
	\noindent {\bf Case (8)} 
	If $v_1,v_2,v_{n-1},v_n\in D$ and $v_3,v_{n-2}\not\in D$, then ${\cal D}_a(C_n,i)={\cal D}(P_{n-6},i-4)\cup \{v_1,v_2,v_{n-1},v_n\}$. In this case $d_a(C_n,i)$ is $d(P_{n-6},i-4)$. 
	
	\medskip
	\noindent {\bf Case (9)}
	If $v_1,v_{n-2},v_{n-1},v_n\in D$ and $v_1,v_{n-3}\not\in D$, then ${\cal D}_a(C_n,i)={\cal D}(P_{n-6},i-4)\cup \{v_1,v_{n-2},v_{n-1},v_n\}$. In this case $d_a(C_n,i)$ is $d(P_{n-6},i-4)$.		
	
		\medskip
	\noindent {\bf Case (10)}
	If $v_2,v_3,v_4,v_5\in D$ and $v_1,v_{6}\not\in D$, then ${\cal D}_a(C_n,i)={\cal D}(P_{n-6},i-4)\cup \{vv_2,v_3,v_4,v_5\}$. In this case $d_a(C_n,i)$ is $d(P_{n-6},i-4)$.
	
		\medskip
	\noindent {\bf Case (11)}
	If $v_{n-3},v_{n-2},v_{n-1},v_n\in D$ and $v_1,v_{n-4}\not\in D$, then ${\cal D}_a(C_n,i)={\cal D}(P_{n-6},i-4)\cup \{v_{n-3},v_{n-2},v_{n-1},v_n\}$. In this case $d_a(C_n,i)$ is $d(P_{n-6},i-4)$. 									
	
	\medskip
	By continuing these steps, we will have the following $i+1$ end steps: 
	
	\medskip
	\noindent {\bf  Case ($1'$)}
	If $v_1,v_2,v_3,\ldots,v_{i-1}\in D$ and $v_i,v_n\not\in D$, then
	$${\cal D}_a(C_n,i)={\cal D}(P_{n-i-1},1)\cup \{v_1,v_2,v_3,\ldots,v_{i-1}\}.$$
	In this case the number of isolate dominating sets of $C_n$ with cardinality $i$ is $d(P_{n-i-1},1)$. 	
	
	\medskip
	\noindent {\bf Case ($2'$)}
	If $v_n,v_1,v_2,v_3,\ldots,v_{i-2}\in D,$ and $v_i,v_{n-1}\not\in D$, then 
	$${\cal D}_a(C_n,i)={\cal D}(P_{n-i-1},1)\cup \{v_n,v_1,v_2,v_3,\ldots,v_{i-2}\}.$$
	In this case the number of isolate dominating sets of $C_n$ with cardinality $i$ is $d(P_{n-i-1},1)$. 
	
	\medskip
	\noindent {\bf Case($i'$)}
	If $v_2,v_3,v_4,\ldots,v_{i}\in D,$ and $v_1,v_{i+1}\not\in D$, then 
	$${\cal D}_a(C_n,i)={\cal D}(P_{n-i-1},1)\cup \{v_{n-i+2},v_{n-i+3},\ldots,v_n,v_1\}.$$
	In this case the number of isolate dominating sets of $C_n$ with cardinality $i$ is $d(P_{n-i-1},1)$.
	\medskip	
	\noindent {\bf Case ($(i+1)'$)} 
	If $v_{n-i+1},v_{n-i+2},v_{n-i+3},\ldots,v_{n}\in D,$ and $v_1,v_{n-i}\not\in D$, then 
	$${\cal D}_a(C_n,i)={\cal D}(P_{n-i-1},1)\cup \{v_2,v_3,v_4,\ldots,v_i\}.$$
	In this case the number of isolate dominating sets of $C_n$ with cardinality $i$ is $d(P_{n-i-1},1)$.
	
	\medskip
	So we have:
	\begin{align*}
	d_a(C_n,i)  \geq  5d(P_{n-5},i-3) + 6d(P_{n-6},i-4) + \ldots + (i+1)d(P_{n-i-1},1)
	\end{align*}
	and therefore we have the result.\qed	
	\end{proof} 

		\begin{figure}
		\begin{center}
			\psscalebox{0.45 0.45}
{
\begin{pspicture}(0,-11.47)(16.9,8.89)
\psdots[linecolor=black, dotsize=0.4](0.45,8.18)
\psdots[linecolor=black, dotsize=0.4](1.25,8.18)
\psdots[linecolor=black, dotsize=0.4](2.05,8.18)
\psdots[linecolor=black, dotsize=0.4](2.85,8.18)
\psdots[linecolor=black, dotsize=0.4](3.65,8.18)
\psdots[linecolor=black, dotsize=0.4](4.45,8.18)
\psline[linecolor=black, linewidth=0.08](0.05,6.18)(5.25,6.18)(5.25,6.18)
\psline[linecolor=black, linewidth=0.08](0.05,5.38)(0.05,5.38)(5.25,5.38)(5.25,5.38)
\psline[linecolor=black, linewidth=0.08](0.05,4.58)(5.25,4.58)(5.25,4.58)
\psline[linecolor=black, linewidth=0.08](0.05,3.78)(5.25,3.78)(5.25,3.78)
\psdots[linecolor=black, dotsize=0.4](0.45,5.78)
\psdots[linecolor=black, dotsize=0.4](1.25,4.98)
\psdots[linecolor=black, dotsize=0.4](0.45,4.18)
\psline[linecolor=black, linewidth=0.08](0.05,0.58)(5.25,0.58)(5.25,0.58)
\psline[linecolor=black, linewidth=0.08](0.05,-0.22)(5.25,-0.22)(5.25,-0.22)
\psdots[linecolor=black, dotsize=0.4](0.45,1.78)
\psdots[linecolor=black, dotsize=0.4](1.25,1.78)
\psdots[linecolor=black, dotsize=0.4](2.05,1.78)
\psdots[linecolor=black, dotsize=0.4](2.85,1.78)
\psline[linecolor=black, linewidth=0.08](0.05,-1.82)(5.25,-1.82)(5.25,-1.82)
\psline[linecolor=black, linewidth=0.08](0.05,-1.02)(5.25,-1.02)(5.25,-1.02)
\psline[linecolor=black, linewidth=0.08](4.45,-2.62)(5.25,-2.62)(5.25,-2.62)
\psdots[linecolor=black, dotsize=0.1](1.25,-3.02)
\psdots[linecolor=black, dotsize=0.1](1.25,-3.42)
\psdots[linecolor=black, dotsize=0.1](1.25,-3.82)
\psline[linecolor=black, linewidth=0.08](0.05,-9.02)(0.05,-9.82)(5.25,-9.82)(5.25,-9.82)
\psdots[linecolor=black, dotsize=0.4](1.25,-4.62)
\psdots[linecolor=black, dotsize=0.4](2.05,-4.62)
\psdots[linecolor=black, dotsize=0.4](2.85,-4.62)
\psdots[linecolor=black, dotsize=0.4](1.25,-10.22)
\psdots[linecolor=black, dotsize=0.4](2.05,-10.22)
\psdots[linecolor=black, dotsize=0.4](2.85,-10.22)
\psline[linecolor=black, linewidth=0.08](0.05,-9.82)(0.05,-10.62)(5.25,-10.62)(5.25,-10.62)
\psdots[linecolor=black, dotsize=0.4](15.65,8.18)
\psdots[linecolor=black, dotsize=0.4](16.45,8.18)
\psline[linecolor=black, linewidth=0.08](4.85,-9.82)(12.05,-9.82)(12.05,-9.82)
\psline[linecolor=black, linewidth=0.08](5.25,-4.22)(12.05,-4.22)(12.05,-4.22)
\psline[linecolor=black, linewidth=0.08](5.25,-2.62)(12.05,-2.62)(12.05,-2.62)
\psline[linecolor=black, linewidth=0.08](5.25,-1.82)(12.05,-1.82)(12.05,-1.82)
\psline[linecolor=black, linewidth=0.08](5.25,-1.02)(12.05,-1.02)(12.05,-1.02)
\psline[linecolor=black, linewidth=0.08](5.25,-0.22)(12.05,-0.22)(12.05,-0.22)
\psline[linecolor=black, linewidth=0.08](5.25,0.58)(5.25,0.58)(5.25,0.58)(5.25,0.58)(12.05,0.58)(12.05,0.58)
\psline[linecolor=black, linewidth=0.08](5.25,3.78)(12.05,3.78)(12.05,3.78)
\psline[linecolor=black, linewidth=0.08](5.25,4.58)(12.05,4.58)(12.05,4.58)
\psline[linecolor=black, linewidth=0.08](4.85,5.38)(12.05,5.38)(12.05,5.38)
\rput[bl](0.27,8.58){${v_1}$}
\rput[bl](1.09,8.6){${v_2}$}
\rput[bl](1.83,8.6){${v_3}$}
\rput[bl](2.65,8.62){${v_4}$}
\rput[bl](3.51,8.64){${v_5}$}
\rput[bl](4.29,8.62){${v_6}$}
\rput[bl](16.31,8.46){${v_n}$}
\rput[bl](15.33,8.42){${v_{n-1}}$}
\psdots[linecolor=black, dotsize=0.4](7.65,-4.62)
\psdots[linecolor=black, fillstyle=solid,fillcolor=black, dotsize=0.4](8.45,-4.62)
\psline[linecolor=black, linewidth=0.08](9.25,-4.22)(9.25,-5.02)(9.25,-5.02)
\psline[linecolor=black, linewidth=0.08](9.65,-9.82)(9.65,-10.62)(9.65,-10.62)
\psdots[linecolor=black, fillstyle=solid,fillcolor=black, dotsize=0.4](0.45,-10.22)
\psdots[linecolor=black, dotsize=0.4](0.45,-4.62)
\psline[linecolor=black, linewidth=0.06](0.45,8.18)(4.85,8.18)(4.85,8.18)
\psline[linecolor=black, linewidth=0.06](15.25,8.18)(16.45,8.18)(16.45,8.18)
\psdots[linecolor=black, dotsize=0.4](9.25,8.18)
\psdots[linecolor=black, dotsize=0.4](8.45,8.18)
\psdots[linecolor=black, dotsize=0.4](7.65,8.18)
\psline[linecolor=black, linewidth=0.06](7.25,8.18)(9.65,8.18)(9.65,8.18)
\rput[bl](8.29,8.52){${v_i}$}
\rput[bl](7.37,8.52){${v_{i-1}}$}
\rput[bl](8.97,8.5){${v_{i+1}}$}
\psline[linecolor=black, linewidth=0.08](4.45,-2.62)(0.05,-2.62)(0.05,-2.62)
\psline[linecolor=black, linewidth=0.08](0.05,1.38)(12.05,1.38)(12.05,1.38)
\psline[linecolor=black, linewidth=0.08](0.05,2.18)(12.05,2.18)(12.05,2.18)
\psline[linecolor=black, linewidth=0.08](0.05,2.98)(12.05,2.98)(12.05,2.98)
\psline[linecolor=black, linewidth=0.08](5.25,6.18)(12.05,6.18)(11.65,6.18)
\psdots[linecolor=black, dotsize=0.4](14.85,8.18)
\psdots[linecolor=black, dotsize=0.4](14.05,8.18)
\psline[linecolor=black, linewidth=0.08](13.65,8.18)(15.65,8.18)(15.25,8.18)
\psline[linecolor=black, linewidth=0.08](0.45,8.18)(0.45,6.98)(16.45,6.98)(16.45,8.18)(16.05,8.18)
\rput[bl](14.41,8.44){$v_{n-2}$}
\rput[bl](13.55,8.42){$v_{n-3}$}
\psline[linecolor=black, linewidth=0.08](12.05,5.38)(16.85,5.38)(16.85,5.38)
\psline[linecolor=black, linewidth=0.08](11.65,4.58)(16.85,4.58)(16.85,4.58)
\psline[linecolor=black, linewidth=0.08](11.65,3.78)(16.85,3.78)(16.85,3.78)
\psline[linecolor=black, linewidth=0.08](12.05,2.98)(16.85,2.98)(16.85,2.98)
\psline[linecolor=black, linewidth=0.08](12.05,2.18)(16.85,2.18)(16.45,2.18)
\psline[linecolor=black, linewidth=0.08](11.65,1.38)(16.85,1.38)(16.85,1.38)
\psline[linecolor=black, linewidth=0.08](12.05,0.58)(16.85,0.58)(16.85,0.58)
\psline[linecolor=black, linewidth=0.08](12.05,-0.22)(16.85,-0.22)(16.45,-0.22)
\psline[linecolor=black, linewidth=0.08](12.05,-1.02)(16.45,-1.02)(16.85,-1.02)(16.85,-1.02)
\psline[linecolor=black, linewidth=0.08](12.05,-1.82)(16.85,-1.82)(16.85,-1.82)
\psline[linecolor=black, linewidth=0.08](12.05,-2.62)(16.85,-2.62)(16.85,-2.62)
\psdots[linecolor=black, dotsize=0.4](0.45,0.98)
\psdots[linecolor=black, dotsize=0.4](1.25,0.98)
\psdots[linecolor=black, dotsize=0.4](0.45,0.18)
\psline[linecolor=black, linewidth=0.08](0.05,-2.62)(0.05,-5.82)(0.05,-5.82)
\psline[linecolor=black, linewidth=0.08](16.85,-2.62)(16.85,-7.02)(16.85,-7.02)
\psline[linecolor=black, linewidth=0.08](0.05,-5.02)(16.85,-5.02)(16.85,-5.02)
\psline[linecolor=black, linewidth=0.08](5.25,-4.22)(0.05,-4.22)(0.05,-4.22)
\psline[linecolor=black, linewidth=0.08](12.05,-4.22)(16.85,-4.22)(16.85,-4.22)
\psdots[linecolor=black, fillstyle=solid,fillcolor=black, dotsize=0.4](16.45,-4.62)
\psdots[linecolor=black, fillstyle=solid,fillcolor=black, dotsize=0.4](15.65,-5.42)
\psdots[linecolor=black, fillstyle=solid,fillcolor=black, dotsize=0.4](14.85,-6.22)
\psdots[linecolor=black, fillstyle=solid,fillcolor=black, dotsize=0.4](14.05,-7.02)
\psline[linecolor=black, linewidth=0.08](0.05,-5.82)(16.85,-5.82)(16.85,-5.82)
\psline[linecolor=black, linewidth=0.08](16.85,-6.62)(0.05,-6.62)(0.05,-5.82)(0.05,-5.82)
\psline[linecolor=black, linewidth=0.08](16.85,-7.02)(16.85,-7.42)(0.05,-7.42)(0.05,-6.62)(0.45,-6.62)
\psdots[linecolor=black, dotsize=0.4](1.25,-5.42)
\psdots[linecolor=black, dotsize=0.4](2.05,-5.42)
\psdots[linecolor=black, dotsize=0.4](2.85,-5.42)
\psdots[linecolor=black, dotsize=0.4](0.45,-5.42)
\psdots[linecolor=black, dotsize=0.4](1.25,-6.22)
\psdots[linecolor=black, dotsize=0.4](2.05,-6.22)
\psdots[linecolor=black, dotsize=0.4](2.85,-6.22)
\psdots[linecolor=black, dotsize=0.4](0.45,-6.22)
\psdots[linecolor=black, dotsize=0.4](1.25,-7.02)
\psdots[linecolor=black, dotsize=0.4](2.05,-7.02)
\psdots[linecolor=black, dotsize=0.4](2.85,-7.02)
\psdots[linecolor=black, dotsize=0.4](0.45,-7.02)
\psdots[linecolor=black, dotsize=0.4](6.85,-4.62)
\psdots[linecolor=black, dotsize=0.4](6.05,-4.62)
\psdots[linecolor=black, fillstyle=solid,fillcolor=black, dotsize=0.4](7.65,-5.42)
\psdots[linecolor=black, fillstyle=solid,fillcolor=black, dotsize=0.4](6.85,-6.22)
\psdots[linecolor=black, fillstyle=solid,fillcolor=black, dotsize=0.4](6.05,-7.02)
\psdots[linecolor=black, dotsize=0.1](8.45,-7.82)
\psdots[linecolor=black, dotsize=0.1](8.45,-8.22)
\psdots[linecolor=black, dotsize=0.1](8.45,-8.62)
\psline[linecolor=black, linewidth=0.08](0.05,-7.42)(0.05,-9.02)(16.85,-9.02)(16.85,-7.02)(16.85,-7.02)
\psline[linecolor=black, linewidth=0.08](4.85,-10.62)(16.85,-10.62)(16.85,-9.02)(16.85,-9.82)(12.05,-9.82)(11.65,-9.82)
\psdots[linecolor=black, dotsize=0.4](0.45,-9.42)
\psdots[linecolor=black, dotsize=0.4](16.45,-9.42)
\psdots[linecolor=black, dotsize=0.4](15.65,-9.42)
\psdots[linecolor=black, dotsize=0.1](13.65,-9.42)
\psdots[linecolor=black, dotsize=0.4](13.25,-9.42)
\psdots[linecolor=black, dotsize=0.1](14.05,-9.42)
\psdots[linecolor=black, dotsize=0.4](12.45,-9.42)
\psdots[linecolor=black, fillstyle=solid,fillcolor=black, dotsize=0.4](11.65,-9.42)
\psline[linecolor=black, linewidth=0.08](13.65,-6.62)(13.65,-7.42)(13.65,-7.42)
\psline[linecolor=black, linewidth=0.08](8.45,-5.02)(8.45,-5.82)(7.65,-5.82)(7.65,-6.62)(6.85,-6.62)(6.85,-7.42)(6.85,-7.42)
\psline[linecolor=black, linewidth=0.08](16.05,-4.22)(16.05,-5.02)(15.65,-5.02)(15.25,-5.02)(15.25,-5.82)(14.45,-5.82)(14.45,-6.62)(14.45,-6.62)
\psline[linecolor=black, linewidth=0.08](11.25,-9.02)(11.25,-9.82)(11.25,-9.82)
\psline[linecolor=black, linewidth=0.08](1.65,-9.82)(1.65,-9.02)(1.65,-9.02)
\psdots[linecolor=black, fillstyle=solid,fillcolor=black, dotsize=0.4](1.25,-9.42)
\psdots[linecolor=black, dotsize=0.4](16.45,-5.42)
\psdots[linecolor=black, dotsize=0.4](15.65,-6.22)
\psdots[linecolor=black, dotsize=0.4](16.45,-6.22)
\psdots[linecolor=black, dotsize=0.4](14.85,-7.02)
\psdots[linecolor=black, dotsize=0.4](15.65,-7.02)
\psdots[linecolor=black, dotsize=0.4](16.45,-7.02)
\psdots[linecolor=black, dotsize=0.4](6.85,-5.42)
\psdots[linecolor=black, dotsize=0.4](6.05,-6.22)
\psdots[linecolor=black, dotsize=0.4](6.05,-5.42)
\psdots[linecolor=black, dotsize=0.4](5.25,-7.02)
\psdots[linecolor=black, dotsize=0.4](5.25,-6.22)
\psdots[linecolor=black, dotsize=0.4](5.25,-5.42)
\psdots[linecolor=black, dotsize=0.4](5.25,-4.62)
\psdots[linecolor=black, dotsize=0.1](3.65,-4.62)
\psdots[linecolor=black, dotsize=0.1](4.05,-4.62)
\psdots[linecolor=black, dotsize=0.1](4.45,-4.62)
\psdots[linecolor=black, dotsize=0.1](3.65,-5.42)
\psdots[linecolor=black, dotsize=0.1](4.05,-5.42)
\psdots[linecolor=black, dotsize=0.1](4.45,-5.42)
\psdots[linecolor=black, dotsize=0.1](3.65,-6.22)
\psdots[linecolor=black, dotsize=0.1](4.05,-6.22)
\psdots[linecolor=black, dotsize=0.1](4.45,-6.22)
\psdots[linecolor=black, dotsize=0.1](3.65,-7.02)
\psdots[linecolor=black, dotsize=0.1](4.05,-7.02)
\psdots[linecolor=black, dotsize=0.1](4.45,-7.02)
\psdots[linecolor=black, dotsize=0.1](3.65,-10.22)
\psdots[linecolor=black, dotsize=0.1](4.05,-10.22)
\psdots[linecolor=black, dotsize=0.1](4.45,-10.22)
\psdots[linecolor=black, dotsize=0.4](5.25,-10.22)
\psdots[linecolor=black, dotsize=0.4](6.05,-10.22)
\psdots[linecolor=black, dotsize=0.4](6.85,-10.22)
\psdots[linecolor=black, dotsize=0.4](7.65,-10.22)
\psdots[linecolor=black, dotsize=0.4](8.45,-10.22)
\psdots[linecolor=black, dotsize=0.4](14.85,-9.42)
\psdots[linecolor=black, dotsize=0.1](14.45,-9.42)
\psdots[linecolor=black, dotsize=0.1](12.45,8.18)
\psdots[linecolor=black, dotsize=0.1](12.85,8.18)
\psdots[linecolor=black, dotsize=0.1](13.25,8.18)
\psdots[linecolor=black, dotsize=0.1](10.85,8.18)
\psdots[linecolor=black, dotsize=0.1](10.45,8.18)
\psdots[linecolor=black, dotsize=0.1](10.05,8.18)
\psdots[linecolor=black, dotsize=0.1](6.45,8.18)
\psdots[linecolor=black, dotsize=0.1](6.05,8.18)
\psdots[linecolor=black, dotsize=0.1](5.65,8.18)
\psdots[linecolor=black, dotsize=0.4](11.65,8.18)
\psline[linecolor=black, linewidth=0.08](11.25,8.18)(12.05,8.18)(11.65,8.18)(11.65,8.18)
\rput[bl](11.11,8.5){$v_{n-i+1}$}
\psdots[linecolor=black, dotsize=0.4](1.25,5.78)
\psdots[linecolor=black, dotsize=0.4](2.05,5.78)
\psdots[linecolor=black, dotsize=0.4](0.45,4.98)
\psdots[linecolor=black, dotsize=0.4](16.45,4.98)
\psdots[linecolor=black, dotsize=0.4](16.45,4.18)
\psdots[linecolor=black, dotsize=0.4](15.65,4.18)
\psdots[linecolor=black, dotsize=0.4](1.25,3.38)
\psdots[linecolor=black, dotsize=0.4](2.05,3.38)
\psdots[linecolor=black, dotsize=0.4](2.85,3.38)
\psdots[linecolor=black, dotsize=0.4](16.45,2.58)
\psdots[linecolor=black, dotsize=0.4](15.65,2.58)
\psdots[linecolor=black, dotsize=0.4](14.85,2.58)
\psdots[linecolor=black, dotstyle=o, dotsize=0.4, fillcolor=white](2.85,5.78)
\psdots[linecolor=black, dotstyle=o, dotsize=0.4, fillcolor=white](2.05,4.98)
\psdots[linecolor=black, dotstyle=o, dotsize=0.4, fillcolor=white](1.25,4.18)
\psdots[linecolor=black, dotstyle=o, dotsize=0.4, fillcolor=white](16.45,5.78)
\psdots[linecolor=black, dotstyle=o, dotsize=0.4, fillcolor=white](15.65,4.98)
\psdots[linecolor=black, dotstyle=o, dotsize=0.4, fillcolor=white](14.85,4.18)
\psdots[linecolor=black, dotstyle=o, dotsize=0.4, fillcolor=white](0.45,3.38)
\psdots[linecolor=black, dotstyle=o, dotsize=0.4, fillcolor=white](3.65,3.38)
\psdots[linecolor=black, dotstyle=o, dotsize=0.4, fillcolor=white](14.05,2.58)
\psdots[linecolor=black, dotstyle=o, dotsize=0.4, fillcolor=white](0.45,2.58)
\psdots[linecolor=black, dotstyle=o, dotsize=0.4, fillcolor=white](3.65,1.78)
\psdots[linecolor=black, dotstyle=o, dotsize=0.4, fillcolor=white](2.85,0.98)
\psdots[linecolor=black, dotstyle=o, dotsize=0.4, fillcolor=white](2.05,0.18)
\psline[linecolor=black, linewidth=0.08](3.25,6.18)(3.25,5.38)(2.45,5.38)(2.45,4.58)(1.65,4.58)(1.65,3.78)(1.65,3.78)
\psline[linecolor=black, linewidth=0.08](4.05,3.78)(4.05,2.98)(4.05,2.98)
\psline[linecolor=black, linewidth=0.08](16.05,6.18)(16.05,5.38)(16.45,5.38)
\psline[linecolor=black, linewidth=0.08](15.25,5.38)(15.25,4.58)(15.25,4.58)
\psline[linecolor=black, linewidth=0.08](14.45,4.58)(14.45,3.78)(14.45,3.78)
\psline[linecolor=black, linewidth=0.08](13.65,2.98)(13.65,2.18)(13.65,2.18)
\psline[linecolor=black, linewidth=0.08](4.05,2.18)(4.05,1.38)(4.05,1.38)
\psline[linecolor=black, linewidth=0.08](16.05,2.18)(16.05,1.38)(15.25,1.38)(15.25,0.58)(14.45,0.58)(14.45,-0.22)(13.65,-0.22)(13.65,-1.02)(13.65,-1.02)
\psline[linecolor=black, linewidth=0.08](12.85,-1.82)(12.85,-2.62)(12.85,-2.62)
\psline[linecolor=black, linewidth=0.08](4.85,-1.02)(4.85,-1.82)(4.85,-1.82)
\psdots[linecolor=black, dotsize=0.4](2.05,0.98)
\psdots[linecolor=black, dotsize=0.4](1.25,0.18)
\psdots[linecolor=black, dotsize=0.4](0.45,-0.62)
\psdots[linecolor=black, dotsize=0.4](3.65,-1.42)
\psdots[linecolor=black, dotsize=0.4](2.85,-1.42)
\psdots[linecolor=black, dotsize=0.4](2.05,-1.42)
\psdots[linecolor=black, dotsize=0.4](1.25,-1.42)
\psdots[linecolor=black, dotsize=0.4](14.05,-2.22)
\psdots[linecolor=black, dotsize=0.4](14.85,-2.22)
\psdots[linecolor=black, dotsize=0.4](15.65,-2.22)
\psdots[linecolor=black, dotsize=0.4](16.45,-2.22)
\psdots[linecolor=black, dotstyle=o, dotsize=0.4, fillcolor=white](1.25,-0.62)
\psdots[linecolor=black, dotstyle=o, dotsize=0.4, fillcolor=white](4.45,-1.42)
\psdots[linecolor=black, dotstyle=o, dotsize=0.4, fillcolor=white](0.45,-1.42)
\psdots[linecolor=black, dotstyle=o, dotsize=0.4, fillcolor=white](16.45,1.78)
\psdots[linecolor=black, dotstyle=o, dotsize=0.4, fillcolor=white](15.65,0.98)
\psdots[linecolor=black, dotstyle=o, dotsize=0.4, fillcolor=white](14.85,0.18)
\psdots[linecolor=black, dotstyle=o, dotsize=0.4, fillcolor=white](14.05,-0.62)
\psdots[linecolor=black, dotstyle=o, dotsize=0.4, fillcolor=white](13.25,-2.22)
\psdots[linecolor=black, dotstyle=o, dotsize=0.4, fillcolor=white](0.45,-2.22)
\psline[linecolor=black, linewidth=0.08](0.85,2.98)(0.85,2.18)(0.85,2.18)
\psline[linecolor=black, linewidth=0.08](0.85,-1.82)(0.85,-2.62)(0.85,-2.62)
\psdots[linecolor=black, dotsize=0.4](16.45,0.98)
\psdots[linecolor=black, dotsize=0.4](16.45,0.18)
\psdots[linecolor=black, dotsize=0.4](15.65,0.18)
\psdots[linecolor=black, dotsize=0.4](16.45,-0.62)
\psdots[linecolor=black, dotsize=0.4](15.65,-0.62)
\psdots[linecolor=black, dotsize=0.4](14.85,-0.62)
\psline[linecolor=black, linewidth=0.08](0.05,-3.02)(0.05,6.18)(16.85,6.18)(16.85,-3.02)(16.85,-3.02)
\psline[linecolor=black, linewidth=0.08](0.05,-10.62)(0.05,-11.42)(16.85,-11.42)(16.85,-10.22)(16.85,-10.22)
\psline[linecolor=black, linewidth=0.08](0.85,-10.62)(0.85,-11.42)(0.85,-11.42)
\psline[linecolor=black, linewidth=0.08](10.45,-10.62)(10.45,-11.42)(10.45,-11.42)
\psdots[linecolor=black, dotsize=0.4](11.65,-11.02)
\psdots[linecolor=black, dotsize=0.4](12.45,-11.02)
\psdots[linecolor=black, dotsize=0.4](13.25,-11.02)
\psdots[linecolor=black, dotsize=0.4](14.85,-11.02)
\psdots[linecolor=black, dotsize=0.4](15.65,-11.02)
\psdots[linecolor=black, dotsize=0.4](16.45,-11.02)
\psdots[linecolor=black, dotsize=0.1](13.65,-11.02)
\psdots[linecolor=black, dotsize=0.1](14.05,-11.02)
\psdots[linecolor=black, dotsize=0.1](14.45,-11.02)
\psdots[linecolor=black, dotstyle=o, dotsize=0.4, fillcolor=white](10.85,-11.02)
\psdots[linecolor=black, dotstyle=o, dotsize=0.4, fillcolor=white](0.45,-11.02)
\psdots[linecolor=black, dotstyle=o, dotsize=0.4, fillcolor=white](9.25,-10.22)
\psdots[linecolor=black, dotstyle=o, dotsize=0.4, fillcolor=white](8.45,-4.62)
\psdots[linecolor=black, dotstyle=o, dotsize=0.4, fillcolor=white](7.65,-5.42)
\psdots[linecolor=black, dotstyle=o, dotsize=0.4, fillcolor=white](6.85,-6.22)
\psdots[linecolor=black, dotstyle=o, dotsize=0.4, fillcolor=white](6.05,-7.02)
\psdots[linecolor=black, dotstyle=o, dotsize=0.4, fillcolor=white](16.45,-4.62)
\psdots[linecolor=black, dotstyle=o, dotsize=0.4, fillcolor=white](15.65,-5.42)
\psdots[linecolor=black, dotstyle=o, dotsize=0.4, fillcolor=white](14.85,-6.22)
\psdots[linecolor=black, dotstyle=o, dotsize=0.4, fillcolor=white](14.05,-7.02)
\psdots[linecolor=black, dotstyle=o, dotsize=0.4, fillcolor=white](11.65,-9.42)
\psdots[linecolor=black, dotstyle=o, dotsize=0.4, fillcolor=white](1.25,-9.42)
\psdots[linecolor=black, dotstyle=o, dotsize=0.4, fillcolor=white](0.45,-10.22)
\end{pspicture}
}
		\end{center}
		\caption{\small Making accurate dominating sets of $C_n$ related to the proof of Theorem \ref{cycle-lower}} \label{cycle}
	\end{figure}
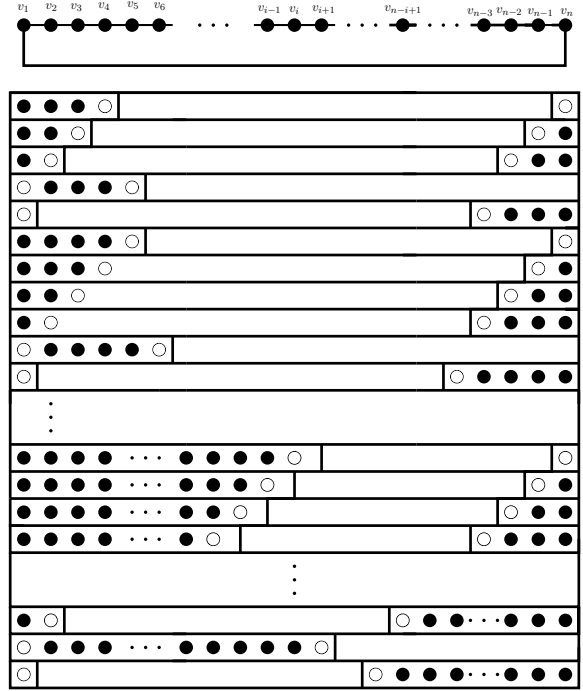

\begin{theorem} \label{cycle-up}
The number of accurate dominating sets of cycle $C_n$, $n\geq 6$, with cardinality $i$,  where $\lfloor\frac{n}{3}\rfloor+2\leq i \leq \lfloor\frac{n}{2}\rfloor$, satisfies: 
	\begin{align*}
	d_a(C_n,i) \leq
	\displaystyle\sum_{k=3}^{i-1}nd(P_{n-k-2},i-k).
	\end{align*}
\end{theorem}

	\begin{proof} 
By Observation \ref{obs}, at least three or more consecutive vertices to have an accurate dominating set for $C_n$. Let $V(C_n)=\{v_1,v_2,....,v_n\}$ and  $D$ be an accurate  dominating  set of $C_n$ with cardinality $i$. First we consider  three consecutive vertices in $D$ which have $n$ cases. For example, if $v_{k-1},v_k,v_{k+1}\in D$ and $v_{k-2},v_{k+2}\not\in D$, then ${\cal D}_a(C_n,i)={\cal D}(P_{n-5},i-3)\cup \{v_{k-1},v_k,v_{k+1}\}$. In this case $d_a(C_n,i)$ is $d(P_{n-5},i-3)$.  So we have $nd(P_{n-5},i-3)$ accurate dominating sets. Now we consider four consecutive vertices in $D$ which are $n$ cases. For example, if $v_{k-1},v_k,v_{k+1},v_{k+2}\in D$ and $v_{k-2},v_{k+3}\not\in D$, then ${\cal D}_a(C_n,i)={\cal D}(P_{n-5},i-3)\cup \{v_{k-1},v_k,v_{k+1},v_{k+2}\}$. In this case $d_a(C_n,i)$ is $d(P_{n-5},i-3)$.  So we have $nd(P_{n-6},i-4)$ accurate dominating sets. By continuing this process, since we count some cases possibly more than once, so we have the result.
 	\qed
	\end{proof} 

	\begin{remark}
	The upper bound in Theorem \ref{cycle-up} is sharp. It suffices to consider the cycle  $C_{10}$ (see Figure \ref{cycle10}) and $i=5$. Then we have $30$ accurate dominating sets as follows:
	\begin{align*}
&\{1,2,3,5,8\} , \quad \{1,2,3,6,8\}  , \quad \{1,2,3,6,9\}  , \quad \{2,3,4,7,9\} , \quad  \{2,3,4,7,10\}  ,\\
  &\{2,3,4,6,9\}  ,\quad \{3,4,5,7,10\}  ,\quad   \{3,4,5,8,10\}  ,\quad   \{3,4,5,8,1\}  ,\quad \{4,5,6,8,1\}  , \\
 &\{4,5,6,9,1\} , \quad  \{4,5,6,9,2\}  ,\quad  \{5,6,7,9,2\}  , \quad  \{5,6,7,10,2\}  , \quad  \{5,6,7,10,3\}  , \\
&\{6,7,8,10,3\}  , \quad  \{6,7,8,1,3\}  ,\quad  \{6,7,8,1,4\}  , \quad \{7,8,9,1,4\}  , \quad  \{7,8,9,2,4\}  , \\
&\{7,8,9,2,5\} , \quad \{8,9,10,2,5\}  , \quad  \{8,9,10,3,5\}  , \quad  \{8,9,10,3,6\}  , \quad \{9,10,1,3,6\}  , \\
&\{9,10,1,4,6\}  , \quad  \{9,10,1,4,7\}  , \quad \{10,1,2,4,7\}  ,\quad   \{10,1,2,5,7\} , \quad  \{10,1,2,5,8\}  .
	\end{align*}
	Since $d(P_5,2)=3$ and $d(P_4,1)=0$, then the equality holds.
	\end{remark}
	
		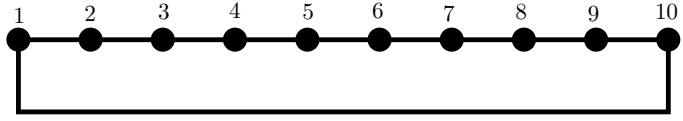
\begin{figure}
		\begin{center}
			\psscalebox{0.8 0.8}
{
\begin{pspicture}(0,-3.51)(11.194231,-1.63)
\psdots[linecolor=black, dotsize=0.4](0.19711548,-2.26)
\psdots[linecolor=black, dotsize=0.4](1.3971155,-2.26)
\psdots[linecolor=black, dotsize=0.4](2.5971155,-2.26)
\psdots[linecolor=black, dotsize=0.4](3.7971156,-2.26)
\psdots[linecolor=black, dotsize=0.4](4.9971156,-2.26)
\psdots[linecolor=black, dotsize=0.4](6.1971154,-2.26)
\psline[linecolor=black, linewidth=0.08](0.19711548,-2.26)(6.1971154,-2.26)(6.1971154,-2.26)
\rput[bl](0.117115475,-1.98){1}
\rput[bl](1.2971154,-1.94){2}
\rput[bl](2.5171156,-1.92){3}
\rput[bl](3.6971154,-1.9){4}
\rput[bl](4.9171157,-1.92){5}
\rput[bl](6.0771155,-1.9){6}
\psdots[linecolor=black, dotsize=0.4](7.3971157,-2.26)
\psdots[linecolor=black, dotsize=0.4](8.5971155,-2.26)
\psdots[linecolor=black, dotsize=0.4](9.797115,-2.26)
\psdots[linecolor=black, dotsize=0.4](10.997115,-2.26)
\psline[linecolor=black, linewidth=0.08](6.1971154,-2.26)(10.997115,-2.26)(10.997115,-3.46)(0.19711548,-3.46)(0.19711548,-2.26)(0.19711548,-2.26)
\rput[bl](7.2571154,-1.94){7}
\rput[bl](8.457115,-1.92){8}
\rput[bl](9.657116,-1.94){9}
\rput[bl](10.777116,-1.92){10}
\end{pspicture}
}
		\end{center}
		\caption{Cycle graph $C_{10}$ with vertex set $V=\{1,2,3,4,5,6,7,8,9,10\}$} \label{cycle10}
	\end{figure}

\begin{theorem} \label{cycle-cycle}
The number of accurate dominating sets of cycle $C_n$, $n\geq 6$, with cardinality $i$, $d_a(C_n,i)$, satisfies: 
	\begin{align*}
	d_a(C_n,i) \geq (n-i+1)d_a(C_{n-1},i-1).
	\end{align*}
\end{theorem}

	\begin{proof} 
	First, we consider the graph $C_{n-1}$ and find all of the accurate dominating sets of that with size $i-1$. Now we have $n-i+1$ vertices which are not in the dominating set. We put a new vertex as a neighbour of them and also in the dominating set. Therefore we have an accurate dominating set with size $i$ for $C_n$. 
 	\qed
	\end{proof} 

	\begin{remark}
	The lower bound in Theorem \ref{cycle-cycle} is sharp. It suffices to consider $i=n$.
	\end{remark}

By Theorems \ref{Llano1} and \ref{Llano2}, we have the following result:

\begin{corollary}\label{compare}
  For every $n\geq 3$ and $i\geq \lfloor\frac{n}{2}\rfloor$, $d_a(P_n,i)\leq d_a(C_n,i)$. 
  \end{corollary}

Comparing  $d_a(P_n,i)$ with $d_a(C_n,i)$ for $\lceil\frac{n}{3}\rceil\leq i\leq \lfloor\frac{n}{2}\rfloor$, looks interesting. Note that The inequality in 
Corollary \ref{compare} is not true for every $i$. For example $d_a(C_9,3)=0<d_a(P_9,3)=1$.

\section{Conclusions} 
We studied the number of accurate dominating sets for certain graphs. For some graphs we found the exact formula for the number of accurate dominating sets of cardinality $i$, but for paths and cycles the problem looks difficult. We presented some inequalities for $d_a(P_n,i)$ and $d_a(C_n,i)$, but until now all attempts to find a formula  for $d_a(P_n,i)$ and $d_a(C_n,i)$ failed. There are recurrence relations for the number of dominating sets of arbitrary graph $G$ with cardinality $i$ (\cite{Kot})  and it is interesting problem to find recurrence relations for the number of accurate dominating sets, too. 
The paper leaves some open problems, among them:

\noindent{\bf Open Problem 1.} 
Find  explicit formulas for $d_a(P_n,i)$ and $d_a(C_n,i)$, for $\lceil\frac{n}{3}\rceil\leq i\leq \lfloor\frac{n}{2}\rfloor$.

\noindent{\bf Open Problem 2.} 
Find recurrence relations and splitting formulas for the $d_a(G,i)$ using simple graph operations.


\end{document}